\documentclass[12pt]{amsart}
\usepackage{a4wide}
\usepackage{amssymb}
\usepackage{amsthm}
\usepackage{amsmath}
\usepackage{amscd}
\usepackage{verbatim}
\usepackage[all]{xy}
\usepackage{longtable, lscape}
\usepackage{todonotes, xcolor}

\numberwithin{equation}{section}

\theoremstyle{plain}
\newtheorem{theorem}{Theorem}[section]
\newtheorem{corollary}[theorem]{Corollary}
\newtheorem{lemma}[theorem]{Lemma}
\newtheorem{proposition}[theorem]{Proposition}

\theoremstyle{definition}

\newtheorem{example}[theorem]{Example}

\theoremstyle{remark}
\newtheorem{remark}[theorem]{Remark}

\newcommand{\R}{\mathbb{R}}
\newcommand{\Q}{\mathbb{Q}}
\newcommand{\Z}{\mathbb{Z}}

\newcommand{\C}{\mathbb{C}}

\renewcommand{\H}{\mathbb{H}}

\newcommand{\G}{\mathbb{G}}

\newcommand{\zxz}[4]{\begin{pmatrix} #1 & #2 \\ #3 & #4 \end{pmatrix}}

\newcommand{\leg}[2]{\left( \frac{#1}{#2} \right)}
\newcommand{\kzxz}[4]{\left(\begin{smallmatrix} #1 & #2 \\ #3 & #4\end{smallmatrix}\right) }
\newcommand{\kabcd}{\kzxz{a}{b}{c}{d}}

\newcommand{\calD}{\mathcal{D}}

\newcommand{\calL}{\mathcal{L}}
\newcommand{\calM}{\mathcal{M}}

\newcommand{\calO}{\mathcal{O}}

\newcommand{\calQ}{\mathcal{Q}}

\newcommand{\calT}{\mathcal{T}}

\newcommand{\frakm}{\mathfrak m}

\newcommand{\eps}{\varepsilon}
\newcommand{\bs}{\backslash}

\newcommand{\tr}{\operatorname{tr}}

\newcommand{\sgn}{\operatorname{sgn}}

\newcommand{\Cl}{\operatorname{Cl}}
\newcommand{\Sl}{\operatorname{SL}}

\newcommand{\SL}{\operatorname{SL}}

\newcommand{\Mp}{\operatorname{Mp}}
\newcommand{\Orth}{\operatorname{O}}

\newcommand{\Hom}{\operatorname{Hom}}
\newcommand{\Aut}{\operatorname{Aut}}
\newcommand{\Mat}{\operatorname{Mat}}

\newcommand{\Iso}{\operatorname{Iso}}
\newcommand{\dv}{\operatorname{div}}

\newcommand{\res}{\operatorname{res}}

\newcommand{\ii}{\text{\rm add}}

\newcommand{\SO}{\operatorname{SO}}
\newcommand{\PSL}{\operatorname{PSL}}

\newcommand{\supp}{\operatorname{supp}}

\newcommand{\Div}{\operatorname{Div}}

\newcommand{\ord}{\operatorname{ord}}

\newcommand{\disc}{\operatorname{disc}}

\newcommand{\mf}{\mathfrak{m}}

\begin{document}

\title[Heegner divisors in generalized Jacobians and singular moduli]{Heegner divisors in generalized Jacobians \\and traces of singular moduli}

\author[Jan H.~Bruinier and Yingkun Li]{Jan
Hendrik Bruinier and Yingkun Li}
\address{Fachbereich Mathematik,
Technische Universit\"at Darmstadt, Schlossgartenstrasse 7, D--64289
Darmstadt, Germany}
\email{bruinier@mathematik.tu-darmstadt.de}
\email{li@mathematik.tu-darmstadt.de}

\thanks{The authors are partially supported by DFG grant BR-2163/4-1.}

\date{\today}

\begin{abstract}
In the present paper we prove an abstract modularity result for classes of Heegner divisors in the generalized Jacobian of a modular curve associated to a cuspidal modulus. Extending the Gross-Kohnen-Zagier theorem, we prove that the generating series of these classes is a weakly holomorphic modular form of weight $3/2$. Moreover, we show that
any harmonic Maass forms of weight $0$ defines a functional on the generalized Jacobian. Combining these results, we obtain a unifying framework and new proofs for the Gross-Kohnen-Zagier theorem and Zagier's modularity of traces of singular moduli, together with new geometric interpretations of the traces with non-positive index.
\end{abstract}

\maketitle


\section{Introduction}
\label{sect:intro}

The celebrated Gross-Kohnen-Zagier theorem \cite{GKZ} states that the generating series of Heegner divisors on the modular curve $X_0(N)$ is a cusp form of weight $3/2$ with values in the Jacobian of $X_0(N)$. This result was later generalized by various authors to orthogonal and unitary Shimura varieties of higher dimension, see e.g.~\cite{Bo2}, \cite{Ku:MSRI}, \cite{Liu}.

In a different direction, Zagier \cite{Za2} proved that the traces of the normalized $j$-invariant over Heegner divisors of discriminant $-d$ on the modular curve $X(1)$ are the coefficients of a {\em weakly} holomorphic modular form of weight $3/2$. This result was also generalized in subsequent work to modular curves of arbitrary level, traces of harmonic Maass forms over twisted Heegner divisors, and to cover more general non-positive weight modular functions, see e.g.~\cite{AE}, \cite{BOR}, \cite{BF2}, \cite{DJ}, \cite{Fu}, \cite{Ki}.
Recently, Gross \cite{Gr} has explained how Zagier's original result can be related to their earlier joint result with Kohnen. He showed that the traces of singular moduli on $X(1)$ can be interpreted in terms of
Heegner divisors in the {\em generalized} Jacobian associated with the modulus $2\cdot (\infty)$.

In the present paper, we pick up Gross' idea
and define classes of Heegner divisors of arbitrary
discriminant in the generalized Jacobian $J_\frakm(X)$ of a modular curve $X$ of arbitrary level with cuspidal modulus $\frakm$. Then we prove that the generating series of these classes is a weakly holomorphic modular form of weight $3/2$ with values in $J_\frakm(X)$. Our argument is a generalization of Borcherds' proof \cite{Bo2} of the Gross-Kohnen-Zagier theorem and relies on the construction of explicit relations among Heegner divisors given by automorphic products. Note that in contrast to \cite{Bo2}, we need to use the explicit infinite product expansions of automorphic products at all cusps of $X$. By applying the natural map between $J_\frakm(X)$ and the usual Jacobian $J(X)$ to this generating series, we recover the `classical' Gross-Kohnen-Zagier theorem.

Then we show that every harmonic Maass form $F$ of weight $0$ on $X$ with vanishing constant term at every cusp (such as the normalized $j$-function when $X=X(1)$), defines a functional $\tr_F$ on $J_\frakm(X)$. The value of $\tr_F$ on Heegner divisors of negative discriminant $-d$ is just the sum of the values of $F$ over the Heegner points of discriminant $-d$. The value of $\tr_F$ on `Heegner divisors' of non-negative discriminant can be explicitly computed in terms of the principal parts of $F$ at the cusps. In that way we are able to recover Zagier's result and its generalisations in \cite{AE}, \cite{BF2}.

We now describe the content of the present paper in more detail. To simplify the exposition, throughout this  introduction we let $p$ be prime or $1$ and consider the modular curve $X_0^*(p)$ associated to the extension $\Gamma_0^*(p)$  of $\Gamma_0(p)$ in $\PSL_2(\Z)$ by the Fricke involution.
In the body of this paper, we consider modular curves of arbitrary level (as modular curves associated to orthogonal groups of signature $(1,2)$).

Let $\infty$ be the cusp of $X_0^*(p)$ and let $m$ be a non-negative integer.
Then $\frakm=m\cdot (\infty)$ is an effective divisor. Recall that the generalized Jacobian $J_\frakm(X_0^*(p))$ of $X_0^*(p)$ associated with the modulus $\frakm$ is a commutative algebraic group
whose  rational points correspond to classes of divisors of degree zero modulo $\frakm$-equivalence, see Section \ref{sect:2} and \cite{Se}. When $m=0$, then $J_\frakm(X_0^*(p))$ is simply the  usual Jacobian.
For any integer $d$, let $\calQ_{p,d}$ be the set of (positive definite if $d>0$) integral binary quadratic forms $[a,b,c]$ of discriminant $-d=b^2-4ac$ with $p$ dividing $a$.
If $d\neq 0$ then $\Gamma_0^*(p)$ acts on $\calQ_{p,d}$ with  finitely many orbits.

If $d$ is positive, then any $Q\in \calQ_{p,d}$ defines a point $\alpha_Q$ in the upper complex half plane $\H$, the solution of the equation $az^2+bz+c=0$ with positive imaginary part. There is a corresponding Heegner divisor of discriminant $-d$ on $X_0^*(p)$
given by
\[
Y(d)= \sum_{Q\in \calQ_{p,d}/\Gamma_0^*(p)} \frac{1}{|\Gamma_0^*(p)_Q|}\cdot  (\alpha_Q),
\]
where $\Gamma_0^*(p)_Q$ is the (finite) stabilizer of $Q$ (see (1.5) in \cite{BF2}).
The divisor
\[
Z(d)= Y(d)-\deg(Y(d))\cdot (\infty)
\]
has degree zero and is defined over $\Q$. We denote by $[Z(d)]_\frakm$ its class in the generalized Jacobian $J_\frakm(X_0^*(p))$.

If $d$ is negative,
any $Q\in \calQ_{p,d}$ defines an oriented
geodesic cycle on $\H\cup P^1(\R)$, given by the equation
$a|z|^2+b\Re(z)+c=0$. It has nontrivial intersection with $P^1(\Q)$ if and only if $d$ is the negative of a square of an integer. In this case the two solutions in $P^1(\Q)$ define cusps of the modular curve. There is a unique cusp $c_Q\in P^1(\Q)$ from which the geodesic originates. (In the present $\Gamma_0^*(p)$-example all cusps collapse to $\infty$ under the map to the quotient, but this is of course not true for more general congruence subgroups.) If $d=-b^2$ for a non-zero  integer $b$, then $Q$ is $\Gamma_0^*(p)$-equivalent to $[0,b,c]$ with $c\in \Z/b\Z$ and $c_Q$ is equivalent to $\infty$. We let $h_Q\in \Q(X_0^*(p))^\times$ be a function satisfying
\[
h_Q= 1-q_\infty^b+O(q_\infty^m)
\]
at the cusp $\infty$, where $q_\infty$ is the uniformizing parameter of the completed local ring at $\infty$ given by the Tate curve over $\Z[[q_\infty]]$.
Then we define
\[
[Z(d)]_\frakm = [\dv(h_{[0,b,0]})]_\frakm=\sum_{Q\in \calQ_{p,d}/\Gamma_0^*(p)} \frac{1}{b}\cdot [\dv(h_Q)]_\frakm.
\]
Note that this class vanishes if $d\leq -m^2$.
If $d<0$ is not
the negative of the square of an integer,
we put $[Z(d)]_\frakm=0$. Finally, for $d=0$, we define $[Z(0)]_\frakm$ as the class of the line bundle of modular forms $\calM_{-1}$ of weight $-1$ on $X_0^*(p)$ (see Section \ref{sect:2} for details).

To describe the relations among the classes $[Z(d)]_\frakm$, we consider the generating series
\[
A_\frakm(\tau) = \sum_{\substack{d\in \Z\\ d>-m^2}} [Z(d)]_\frakm \cdot q^d \in \C((q))\otimes J_\frakm(X_0^*(p)).
\]
It is a formal Laurent series in the variable $q=e^{2\pi i\tau}$ for $\tau \in \H$. Our first main result is the following (see also Theorem \ref{thm:mod}).

\begin{theorem}
\label{thm:intromod}
The generating series $A_\frakm(\tau)$ is a weakly holomorphic modular form of weight $3/2$ for the group $\Gamma_0(4p)$, that is,
$A_\frakm(\tau)\in M_{3/2}^!(\Gamma_0(4p))\otimes J_\frakm(X_0^*(p))$.
\end{theorem}

Under the natural map
\[
J_\frakm(X_0^*(p))\longrightarrow J(X_0^*(p))
\]
the classes $[Z(d)]_\frakm$ with $d\leq 0$ are mapped to zero. Applying it to $A_\frakm(\tau)$, we recover the Gross-Kohnen-Zagier theorem (see also Corollary \ref{cor:gkz}).

\begin{corollary}[Gross-Kohnen-Zagier]
The generating series $A_0(\tau)$ of classes of Heegner divisors $[Z(d)]_0$
in the Jacobian
is a cusp form of weight $3/2$ for the group $\Gamma_0(4p)$, that is,
$A_0(\tau)\in S_{3/2}(\Gamma_0(4p))\otimes J(X_0^*(p))$.
\end{corollary}

To recover the results of \cite{Za2} and \cite{BF2} on traces of modular functions from Theorem \ref{thm:intromod}, we show that harmonic Maass forms define functionals on $J_\frakm(X_0^*(p))$. Let $F\in H_0^+(\Gamma_0^*(p))$ be a harmonic Maass form for $\Gamma_0^*(p)$ of weight $0$ as in \cite{BF1}. Denote the Fourier expansion of the holomorphic part of $F$ by
\[
F^+(\tau)=\sum_{n\gg-\infty}c_F^+(n) \cdot q_\infty^n.
\]

\begin{proposition}
Assume that $c_F^+(n)=0$ for $n\leq - m$ and  $c_F^+(0)=0$.
Then there is a linear map
$\tr_F:J_\frakm(X_0^*(p))\to\C$
defined by
\begin{align*}
[D]_\frakm\mapsto \tr_F(D):=\sum_{\substack{a\in \supp(D)\setminus \{\infty\}}} n_a \cdot F(a)
\end{align*}
for divisors $D=\sum_a n_a \cdot(a)$ in $\Div^0(X_0^*(p))$.
\end{proposition}

The images under $\tr_F$ of the classes $[Z(d)]_\frakm$ with $d\leq 0$ can be explicitly computed in terms of the principal part of $F$. As a consequence we derive (see Theorem \ref{thm:tr}):

\begin{theorem}
\label{thm:introtr}
The series $\tr_F(A_\frakm)$ is a weakly holomorphic modular form in $M_{3/2}^!(\Gamma_0(4p))$. It is explicitly given by
\begin{align*}
\tr_F(A_\frakm)&
=\sum_{d>0} F\big(Y(d)\big)\cdot q^d
+\sum_{n\geq 1}c_{F}^+(-n)\big(\sigma_1(n)+p\sigma_1(n/p)\big)
-\sum_{b>0} \sum_{n>0}  c_F^+(-bn)\cdot b\cdot q^{-b^2}.
\end{align*}
\end{theorem}

The modularity of the right hand side was also proved in \cite{BF2} by interpreting it as the Kudla-Millson theta lift of $F$. 
Applying this theorem to the special case where $p=1$, $m\geq 2$, and  $F=j-744$, Zagier's original result on traces of singular moduli can be obtained.

In the body of the paper we work with modular curves of arbitrary level associated with orthogonal groups of even lattices of signature $(1,2)$. This setup is natural, since the proof of Theorem \ref{thm:intromod} implicitly relies on the singular theta correspondence for the dual reductive pair given by $\Sl_2$ and $\Orth(1,2)$.
For the modulus we allow arbitrary effective divisors that are supported on the cusps. The generating series of Heegner divisors is then a vector valued modular form for the metaplectic extension of $\SL_2(\Z)$  transforming with the Weil representation of a finite quadratic module.

This paper is organized as follows. In Section 2 we recall some basic facts on generalized Jacobians of curves. Section 3 contains our setup for modular curves associated to orthogonal groups, Heegner divisors, and vector valued modular forms.
Then we define classes of Heegner divisors in generalized Jacobians in Section 4, and prove the abstract modularity theorem for these classes.
In Section 5 we prove that harmonic Maass forms define functionals on the generalized Jacobian and derive modularity results for the traces of harmonic Maass forms over Heegner divisors from the abstract modularity theorem. We also give some explicit examples and indicate possible generalizations in Section 6.


We thank J.~Funke, B.~Gross,  and S.~Kudla for useful conversations on the content of this paper. Moreover, we thank the referee for his/her valuable comments.


\section{Generalized Jacobians}
\label{sect:2}

Let $X$ be a complete non-singular algebraic curve over a field $k$ of
characteristic $0$. Let $\Div^0(X)$ be the group of divisors of $X$ of
degree $0$ defined over $k$, and denote by $P(X)$ the subgroup of divisors of rational functions $f\in k(X)^\times$.  The
Jacobian $J(X)$ of $X$ is a commutative algebraic group over $k$ whose
$k$-rational points are isomorphic to the quotient group
$\Div^0(X)/P(X)$.

Recall that there is the notion of the generalized Jacobian, see
e.g.~\cite[Chapter 5]{Se} for details.
Let $S\subset X(k)$ be a finite set of points, and for $s\in S$ let $m_s\in \Z_{\geq 0}$.
Then \[
\frakm=\sum_{s\in S} m_s \cdot(s)
\]
is an effective divisor defined over $k$.
Let $\calO_s$ be the
ring of integers in the completion $k(X)_s$ of $k(X)$ at $s$, and let
$\pi_s \in \calO_s$ be a uniformizer.  If $f,g\in k(X)_s$ and $n\in \Z$,
we write
\[
f=g+O(\pi_s^n)
\]
if $f-g\in \pi_s^n\calO_s$.
We consider the subgroup
\[
P_\frakm(X)= \{\dv(f):\; \text{$f\in k(X)^\times$ with
  $\pi_s^{-\ord_s(f)}f=1+O(\pi_s^{m_s})$ for all $s\in S$}\}
\]
of $P(X)$.
The generalized Jacobian
$J_\frakm(X)$ associated with the modulus $\frakm$ is a commutative
algebraic group over $k$, whose $k$-rational points satisfy
\begin{align}
J_\frakm(X)(k) \cong \Div^0(X)/P_\frakm(X).
\end{align}
The quotient on the right hand side is also canonically isomorphic to the subgroup of divisors in $\Div^0(X)$ coprime to $S$ modulo $\frakm$-equivalence.
For a divisor $D\in \Div^0(X)$ we denote by $[D]_\frakm$ the
corresponding class in $J_\frakm(X)(k)$.

There is a canonical rational map $\varphi_\frakm:X\to J_\frakm(X)$
defined over $k$ which is regular outside $S$, see \cite[Chapter 5,
Theorem 1]{Se}.  If $\frakm'$ is another effective divisor on $X$ satisfying
$\frakm\geq \frakm'\geq 0$, there exists
a unique homomorphism $J_\frakm\to J_{\frakm'}$ which is compatible
with $\varphi_\frakm$ and $\varphi_{\frakm'}$. It is surjective and
separable \cite[Chapter 5, Proposition 6]{Se}. In particular, there
exists a surjective homomorphism
\begin{align}
\label{eq:pm}
J_\frakm(X)\to J(X).
\end{align}
Its kernel is isomorphic to
\begin{align}
\label{eq:homker}
H_\frakm=\Big(\prod_{\substack{s\in S\\m_s>0}} \G_m\times \G_a^{m_s-1}\Big)/\G_m,
\end{align}
where the quotient is with respect to the diagonally embedded multiplicative group.
Typical elements of the kernel are obtained, by choosing a pair $(s,n)$ with $s\in S$ and $n>0$ and a function $h_{s,n} \in k(X)^\times$ such that
\begin{align}
\label{eq:hn}
h_{s,n}&=1-\pi_s^n+O(\pi_s^{m_s}),\quad\text{at $s$,}\\
\nonumber
h_{s,n}&=1+O(\pi_t^{m_t}),\quad\text{at all $t\in S\setminus\{s\}$.}
\end{align}
An argument as in \cite[Chapter 5, Proposition 8]{Se}
shows that the `additive part' of $H_\frakm$ is generated by the classes
\begin{align}
[\dv(h_{s,n})]_\frakm
\end{align}
for $s\in S$ and $0< n<m_s$. 
Note that for $n\geq m_s$ the class $[\dv(h_{s,n})]_\frakm$ vanishes.

Let $s_0\in S$ be a fixed base point. If $\calL$ is a line bundle on $X$ which is defined over $k$, and $(\phi_s)_{s\in S}$ is a family of local trivializations of $\calL$ at the points of $S$, we can associate to the pair  $(\calL,(\phi_s))$ a class in $J_\frakm(X)$ as follows.
It is easily seen that there exists a rational section $f$ of $\calL$ such that
\begin{align}
\label{eq:scond}
\phi_s^{-1}f=
\pi_s^{a_s}\cdot (1+O(\pi_s^{m_s}))
\end{align}
for some
$a_s\in \Z$ at every $s\in S$.
Then we define
\begin{align}
\label{eq:cl}
[(\calL,(\phi_s)) ]_\frakm = [\dv(f)-\deg(\calL)\cdot (s_0)]_\frakm\in J_\frakm(X)(k).
\end{align}


\section{Modular curves}
\label{sect:3}

Here we recall the description of modular curves as Shimura varieties associated to orthogonal groups.
We also define classes of Heegner divisors in generalized Jacobians.

Let $(L,Q)$ be an isotropic even lattice of signature $(1,2)$. We
denote by $(x,y)$ 
the bilinear form corresponding to
the quadratic form $Q$, normalized  such that $Q(x)=\frac{1}{2}(x,x)$.
For any commutative ring $R$ we write $L_R=L\otimes_\Z R$. Throughout we fix an orientation on $L_\R$, and write $L'$ for the dual lattice of $L$.
Let
\[
N= \min\{n\in \Z_{>0}:\; \text{$nQ(\lambda)\in \Z$ for all $\lambda\in L'$}\}
\]
be the level of $L$, and denote by $\disc(L)=|L'/L|$ the discriminant
of $L$.
We let $\SO(L)$ be the special orthogonal group of $L$ and write $\SO^+(L)$
for the intersection of $\SO(L)$ with the connected component of the
identity of $\SO(L)(\R)$.
The even Clifford algebra of $L_\Q$ is isomorphic to the matrix
algebra $\Mat_2(\Q)$, which induces an isomorphism
$\operatorname{PGL}_2(\Q)\cong \SO(L)(\Q)$.
We realize the hermitian symmetric space corresponding to $\SO(L)$ as the
domain
\[
\calD=\{z\in L_\C:\; (z,z)=0, (z,\bar z)<0\}/\C^\times.
\]
It decomposes into $2$ connected components. We fix one of these
components and denote it by $\calD^+$.

Let $\Gamma=\Gamma_L$ be the discriminant kernel subgroup of $\SO^+(L)$, that is, the kernel of the natural homomorphism
\begin{align*}
\SO^+(L)\longrightarrow \Aut(L'/L).
\end{align*}
Recall that rescaling the quadratic form by a factor $n$ does not change $\SO^+(L)$ while it replaces the discriminant kernel by the full congruence subgroup of level $n$.
We denote by
\begin{align}
\label{eq:yg}
Y_\Gamma = \Gamma\bs \calD^+
\end{align}
the non-compact modular curve associated with $\Gamma$.

Let $\Iso(L)$ be the set of isotropic lines in $L$
(i.e., primitive isotropic rank $1$ sublattices $I\subset L$). The group $\Gamma$ acts with finitely many orbits on $\Iso(L)$.
We denote by $X_\Gamma$ the compact modular curve
obtained by adding to $Y_\Gamma$ the cusps corresponding to the $\Gamma$-classes of
isotropic lines $I\in \Iso(L)$, see e.g.~\cite{BF2}.
It is well known that $X_\Gamma$ is a
projective algebraic curve which has a canonical model over a cyclotomic field.

As in \cite{BF2}, we choose an orientation on the isotropic lines as follows. We fix one line $I_0\in \Iso(L)$ together with an orientation on $I_0$ given by a basis vector $x_0\in I_{0,\R}$. For any other $I\in \Iso(L)$ we choose a $g\in \SO^+(L)(\R)$ such that $gI_{0,\R}=I_\R$. Then $g x_0\in I_\R$ defines an orientation on $I$, which is independent of the choices of $g$ and $x_0$.


Let $I\subset L$ be a primitive isotropic line and write $c_I\in X_\Gamma$ for the cusp corresponding to $I$. Local coordinates near $c_I$ can be described as follows.
We write  $N_I$ for the positive generator of the ideal $(I,L)\subset \Z$. It is a divisor of $N$.
Throughout, we let $\ell=\ell_I$ be the positive generator of $I$
and fix a vector $\ell'=\ell'_I\in L'$ such that
\begin{align}
\label{eq:ell}
(\ell,\ell')=1.
\end{align}
We let $K$ be the even negative definite lattice
\begin{align}
\label{eq:defK}
K=L\cap \ell^\perp \cap \ell'{}^\perp.
\end{align}
If $\ell_K\in K$ denotes a generator,
then $K$ is isomorphic to $\Z$ equipped with the quadratic
form $x\mapsto Q(\ell_K) x^2$.
The quantity $4Q(\ell_K)$ divides $N$.
The holomorphic map
\begin{align}
\label{eq:dcoord}
\H\longrightarrow \calD,\quad w\mapsto \C^\times \big( w\otimes \ell_K + \ell' - Q(w\otimes \ell_K)\ell-Q(\ell')\ell\big)
\end{align}
is injective and has one of the two connected components of $\calD$ as its image.
Possibly replacing $\ell_K$ by its negative, we may assume that this map is an isomorphism from $\H$ onto $\calD^+$.
It is compatible with the natural actions of $\operatorname{PGL}_2^+(\Q)$ on $\H$ by  fractional linear transformations and on $\calD^+$ via the isomorphism with $\SO^+(L)(\Q)$.
For $\mu\in L_\Q\cap I^\perp$ we consider the Eichler transformation
%
\begin{align}
E_{\ell,\mu}(x)= x +(x,\ell)\mu-(x,\mu)\ell-(x,\ell)Q(\mu)\ell
\end{align}
in $\SO^+(L)(\Q)$. It belongs to $\Gamma$ if $\mu \in K$.

\begin{lemma}
The stabilizer in $\Gamma$ of the primitive isotropic line $I$ is given by
\[
\Gamma_{I} = \{E_{\ell,\mu}:\; \mu \in K\}.
\]
\end{lemma}

\begin{proof}
Let $\gamma\in \Gamma_{I}$. Then $\gamma\ell =\pm \ell$. We first assume that $\gamma\ell = \ell$. Then
\[
u:= \gamma\ell'-\ell'
\]
belongs to $L\cap \ell^\perp$, and $v :=u-(u,\ell')\ell$ belongs
to $K$. It is easily checked that
\begin{align*}
E_{\ell,v}  (\ell)&=\ell,\\
E_{\ell,v}  (\ell')&=\gamma\ell'.
\end{align*}
Hence $\gamma^{-1}E_{\ell,v}$ leaves the vectors $\ell$ and $\ell'$ fixed.
Consequently, it maps the orthogonal complement $K$ to itself, and therefore $\ell_K$ to $\pm \ell_K$. Since $\gamma^{-1}E_{\ell,v}$ has determinant $1$, the sign must be positive and thus $\gamma=E_{\ell,v}$.

We now consider the case $\gamma\ell = -\ell$.
The orthogonal transformation $\sigma$ taking $\ell$ to $-\ell$, and $\ell'$ to $-\ell'$, and $\ell_K$ to itself belongs to $\SO(L)(\Q)$.
The element $\sigma\gamma\in \SO(L)(\Q)$ fixes $\ell$. Arguing as above, we see that it is equal to an Eichler transformation $E_{\ell,u}\in \SO^+(L)(\Q)$.
This implies that $\sigma$ belongs to the connected component of the identity of $\SO(L)(\R)$. But this leads to a contradiction, since the spinor norm of $\sigma$ is negative, showing that the case  $\gamma\ell = -\ell$ cannot occur.
\end{proof}

The action of $\Z$ on $\H$ by translations corresponds to the action of
 $\Gamma_{I}$ on $\calD^+$.
The induced map
\begin{align}
\label{eq:locpar}
\Z\bs \H\longrightarrow \Gamma_{I} \bs \calD^+,
\end{align}
is an isomorphism.
Hence,  $q_I=e^{2\pi i w}$ defines a local parameter at the
cusp $c_I$ of $X_{\Gamma}$.

\begin{example}
In the special case when $N_I=1$, then $4N=-Q(\ell_K)$ and the discriminant kernel subgroup $\Gamma$
is isomorphic
to $\Gamma_0(N/4)$. The curve $X_{\Gamma}$ is isomorphic to $X_0(N/4)$, with $c_I$ corresponding to the cusp at $\infty$, see e.g.~
\cite[Section 2.4]{BO}.
\end{example}

\subsection{The Weil representation}

Let $\Mp_2(\Z)$ be the metaplectic extension of $\SL_2(\Z)$ by $\{\pm 1\}$, realized by the two possible choices of a holomorphic square root of the automorphy factor $c\tau+d$ for $\kabcd\in \Sl_2(\Z)$, see e.g.~\cite{Bo1}, \cite{Ku:Integrals}.

Recall that there is a Weil representation $\omega_L$ of $\Mp_2(\Z)$
on the complex vector space
$S_L$ of functions $L'/L\to \C$ on the discriminant group.
Identifying $S_L$ with the space of Schwartz-Bruhat functions on
$L\otimes \hat\Q$ which are supported on $L'\otimes \hat\Z$ and
translation invariant under $L\otimes\hat\Z$, the representation
$\omega_L$ can be viewed as the restriction of the usual Weil representation of
$\Mp_2(\hat\Q)$ on $L\otimes \hat\Q$ with respect to the standard
additive character of $\hat\Q$, see \cite{Ku:Integrals}. The
representation $\omega_L$ is the complex conjugate of the
representation $\rho_L$ in \cite{Bo1}, \cite{Br1}, \cite{BF2}.
The action of $\Mp_2(\Z)$ on $S_L$ commutes with the natural action of $\Aut(L'/L)$ by translation of the argument.

If $k\in \frac{1}{2}\Z$,  we denote by $M^!_k(\omega_L)$ the space of $S_L$-valued weakly holomorphic modular forms for $\Mp_2(\Z)$ of weight $k$ with representation $\omega_L$.
The subspace of holomorphic modular forms is denoted by $M_k(\omega_L)$.

\subsection{Heegner divisors}
\label{subsec:Heegner}
For any $d\in \Q^\times$, the group $\Gamma$ acts on the set
\[
L'_d=\{\lambda\in L':\; Q(\lambda)=d\}
\]
with finitely many orbits.
For every $\lambda\in L'$ with $Q(\lambda)>0$, the stabilizer $\Gamma_\lambda \subset \Gamma$ of $\lambda$ is finite, and there is a unique point
$z_\lambda\in \calD^+$ which is orthogonal to $\lambda$.
For  $d\in \Q_{>0}$ and $\varphi\in S_L$ we consider Heegner the divisor
\begin{align}
\label{eq:defheeg}
Y(d,\varphi)=\sum_{\substack{\lambda\in L'_d/\Gamma}}\frac{1}{2|\Gamma_\lambda|}
\varphi(\lambda)\cdot (z_\lambda)
\end{align}
on $X_\Gamma$. It is defined over the field of definition of $X_\Gamma$ and has coefficients in the field of
definition of $\varphi$.
Let $I_0\in \Iso(L)$ be a fixed isotropic line.
We define a divisor of degree $0$ on
$X_\Gamma$ by putting
\begin{align}
Z(d,\varphi)=  Y(d,\varphi)-\deg (Y(d,\varphi))\cdot (c_{I_0}).
\end{align}


\section{A generalized Gross-Kohnen-Zagier theorem}
\label{sect:4}

We now consider classes of Heegner divisors in the generalized
Jacobian of the modular curve $X:=X_\Gamma$
as defined in
the previous section.  We let $k\subset \C$ be the number field obtained by adjoining the primitive root of unity $e^{2\pi i/N}$
to the common  field of definition of the canonical model and all cusps of $X$.
Let $S=\{c_I:\; I\in \Iso(L)/\Gamma\}$ be the set of cusps of $X$ and let
\[
\frakm= \sum_{I\in \Iso(L)/\Gamma} m_I\cdot (c_I)
\]
be a fixed effective divisor supported on $S$.
We consider the generalized Jacobian of $X$ associated with the modulus $\frakm$.
For $I\in\Iso(L)$, we take as the uniformizing
parameter in the completed local ring at $c_I$ 
the
parameter $q_I=e^{2\pi i w}$ defined by \eqref{eq:locpar}
(given by the Tate curve over $\Z[[q_I]]$ when $N_I=1$ such that $X_\Gamma\cong X_0(N/4)$).

Since, throughout this section, we are only interested in the
$k$-valued points of the generalized Jacobian, we briefly write
$J_\frakm(X)$ instead of $J_\frakm(X)( k)$.
For every degree zero divisor $D=\sum a_I \cdot (c_I)\in \Div^0(X)$ supported on $S$ and every tuple $r=(r_I)\in \G_m^{|S|}(k)$,
we choose a function $u_{D,r}\in k(X)^\times$ such that
\begin{align}
\label{eq:rdu}
u_{D,r}&=r_I q_I^{a_I}\cdot \left(1+O(q_I^{m_I})\right)
\end{align}
at $c_I$ for $I\in \Iso(L)$. We write $H_{\G_m,\frakm}$ for the subgroup of $J_\frakm(X)$ generated by the classes $[\dv(u_{D,r}) - D]_\frakm$ of all these functions and let
\begin{align}
\label{eq:defja}
J_\frakm^{\ii}(X)=J_\frakm(X)/ H_{\G_m,\frakm}.
\end{align}
By definition we have $J_\frakm^{\ii}(X)=J_\frakm(X)$ when $|S|=1$.
By the Manin-Drinfeld theorem we have $J_\frakm^{\ii}(X)_\Q=J_\frakm(X)_\Q$ when $\frakm=0$. For general $\frakm$ the kernel of the induced homomorphism
\begin{align}
\label{eq:indhom}
J_\frakm^{\ii}(X)_\Q\to J(X)_\Q
\end{align}
is a quotient of the product of the groups $\G_a^{m_I-1}$ for $I\in \Iso(L)/\Gamma$ with $m_I>0$.

For $d\in \Q_{>0}$ and $\varphi\in S_L$ we consider the class
\begin{align}
[Z(d,\varphi)]_\frakm\in
J_\frakm(X)_\C
\end{align}
of the Heegner divisor $Z(d,\varphi)$ in the generalized Jacobian.

Let $\calT$ be the tautological bundle on $X$, and define the line bundle of modular forms of weight $2k$ on $X$ by $\calM_{2k}=\calT^{\otimes k}$. (
Sections of $\calM_{2k}$ correspond to classical elliptic modular forms of weight $2k$ under the isomorphism $\SO(L)(\Q)\cong \operatorname{PGL}_2(\Q)$.)
Recall that $\calT$ is canonically  trivial in small neighborhoods of the cusps. Hence, taking the induced trivializations and putting $s_0=c_{I_0}$ in \eqref{eq:cl}, we obtain a class $[\calM_{k}]_\frakm\in J_\frakm(X)_\Q$.
For $d=0$ we define
\begin{align}
[Z(0,\varphi)]_\frakm
=\varphi(0)\cdot [\calM_{-1}]_\frakm.
\end{align}

We also define classes for $d\in \Q_{<0}$ as follows.
For a vector $\lambda\in L'_d$, the orthogonal complement $\lambda^\perp\subset L_\Q$ is isotropic if and only if $d\in -2\disc(L)(\Q^\times)^2$. In this case there is a unique pair of isotropic lines $I,\tilde I\in \Iso(L)$ such that
$\lambda^\perp=I_\Q\oplus \tilde I_\Q$ and such that the triple $(\lambda,x,\tilde x)$ is a positively oriented basis of $L_\Q$ for positive basis vectors $x\in I$ and $\tilde x\in \tilde I$.
Following \cite{BF2}, we call $I$ {\em the isotropic line associated to $\lambda$} and write $I\sim \lambda$. Note that  $\tilde I$ is the isotropic line associated to $-\lambda$. We define the  {\em $I$-content} 
$n_I(\mu)$ of any $\mu\in L'\cap I^\perp$ as follows: If $Q(\mu)=0$ we put $n_I(\mu)=0$. If $Q(\mu)\neq 0$ we let $n_I(\mu)$ be the unique non-zero integer such that
\begin{align}
\label{eq:defnl}
(\mu,L\cap I^\perp)= n_I(\mu)\cdot  \Z
\end{align}
and $\sgn(n_I(\mu))\cdot \mu\sim I$.

Now, if $d\in -2\disc(L)(\Q^\times)^2$ and $\lambda\in L'_d$, we let $I\in \Iso(L)$ be the isotropic line associated to $\lambda$ and let $\ell'\in L'$ be as in \eqref{eq:ell} such that $(\ell',I)=\Z$.
We choose a function $h_\lambda\in k(X)^\times$ such that
\begin{align}
\label{eq:defhl}
h_{\lambda}&=
1- e^{2\pi i(\lambda,\ell')}q_I^{ n_I(\lambda)} +O(q_I^{m_I}),\quad\text{at the cusp $c_I$,}\\
\nonumber
h_\lambda&= 1 +O(q_J^{m_J}),\quad\text{at all other  cusps $c_J$.}
\end{align}
The existence of $h_\lambda$ follows for instance from the approximation theorem for valuations, see pp.~29 in \cite{Se}. If $(\lambda,\ell')\in \Z$, then $h_{\lambda}$ agrees with the function $h_{c_I,n_I(\lambda)}\in k(X)^\times$ defined in \eqref{eq:hn}.
For $\varphi\in S_L$ we define
\begin{align}
[Z(d,\varphi)]_\frakm
=
\sum_{\substack{\lambda\in L_d'/\Gamma}}
\frac{1}{2n_I(\lambda)}(\varphi(\lambda)+\varphi(-\lambda)\big)
\cdot \left[\dv \left(h^{-1}_{\lambda}\right) \right]_\frakm.
\end{align}
It is easily checked
that this class is independent of the choices of the functions $h_\lambda$.
If $d<0$ and $d\notin -2\disc(L)(\Q^\times)^2$, we put $[Z(d,\varphi)]_\frakm=0$.

Finally, for all $d\in \Q$
we write $[Z(d)]_\frakm$ for the element of
\[
\Hom(S_L,J_\frakm(X)_\C)\cong J_\frakm(X)_\C\otimes S_L^\vee
\]
given by $\varphi\mapsto [Z(d,\varphi)]_\frakm$.

The classes $[Z(d,\varphi)]_\frakm$ with $d<0$
can also be expressed in a slightly different way.
To this end, for an isotropic line $I$ we define
\[
L'_{d,I}=\{\lambda\in L_d' :\; \text{$\lambda\perp I$ and $\lambda \sim I$}\}.
\]

\begin{lemma}
\label{lem:zd00}
For $d<0$ we have
\begin{align*}
[Z(d,\varphi)]_\frakm
=\sum_{I\in \Iso(L)/\Gamma}
\sum_{\substack{\lambda\in L_{d,I}'/I}}
\frac{1}{2}(\varphi(\lambda)+\varphi(-\lambda)\big)
\cdot \left[\dv \left(h^{-1}_{\lambda}\right) \right]_\frakm.
\end{align*}
\end{lemma}

\begin{proof}
If $L_{d,I}$ is non-empty and if we fix $\lambda_0\in L'_{d,I}$, we have
\begin{align*}
L'_{d,I}&=\{\lambda_0+a\ell/N_I : \; a\in \Z\},\\
L'_{d,I}/\Gamma_{I}&=\{\lambda_0+a\ell/N_I: \; a\in \Z/N_I n_I(\lambda_0)\Z\}, \\
L'_{d,I}/{I}&=\{\lambda_0+a\ell/N_I : \; a\in \Z/N_I \Z\}.
\end{align*}
This implies the assertion.
\end{proof}

\subsection{An abstract modularity theorem}
\label{subsec:mod_thm}
To describe the relations in the generalized Jacobian among the classes $[Z(d)]_\frakm$ we form the generating series
\begin{align}
\label{eq:genser}
A_\frakm(\tau)= \sum_{d\in \frac{1}{N}\Z}
[Z(d)]_\frakm\cdot q^{d}\in S_L^\vee((q))\otimes J^\ii_\frakm(X)_\C.
\end{align}
It is a formal Laurent series in the variable\footnote{Confusion with
  the local parameter $q_I$ at the cusp $c_I$ of $X$
  should not be possible.} $q=e^{2\pi
  i\tau}$, where $\tau\in \H$,  with exponents in $\frac{1}{N}\Z$ and  coefficients in
$S_L^\vee \otimes J^\ii_\frakm(X)_\C $.

\begin{theorem}
\label{thm:mod}
The generating series $A_\frakm(\tau)$ is the $q$-expansion of a weakly holomorphic modular form in $M^!_{3/2}(\omega_L^\vee)\otimes J^\ii_\frakm(X)_\C$.
\end{theorem}

To prove this result, we use the following variant of Borcherds' modularity criterion \cite[Theorem 3.1]{Bo2}.
Let  $\rho$ be a finite dimensional representation of $\Mp_2(\Z)$ on a complex vector space $V$ which is trivial on some congruence subgroup.
The stabilizer in $\Mp_2(\Z)$ of the cusp $\infty$ is generated by the elements $T=\left(\kzxz{1}{1}{0}{1},1\right)$ and $Z=\left(\kzxz{-1}{0}{0}{-1},i\right)$.
The hypothesis on $\rho$ implies that some power of $\rho(T)$ is the identity, and therefore all eigenvalues of $\rho(T)$ are roots of unity.
If $g\in M_k^!(\rho)$ is a weakly holomorphic modular form for $\Mp_2(\Z)$ of weight $k\in \frac{1}{2}\Z$ with representation $\rho$, then it has a Fourier expansion
\[
g(\tau)= \sum_{n\in \Q} a(n)\cdot q^n,
\]
where the coefficients $a(n)\in V$ satisfy the conditions
\begin{align}
\label{eq:condt}
\rho(T)a(n)&=e^{2\pi i n}a(n),\\
\label{eq:condz}
\rho(Z)a(n)&= e^{-\pi i k}a(n).
\end{align}
We write $\rho^\vee$ for the representation dual to $\rho$, and denote by $(\cdot,\cdot)$ the natural pairing $V\times V^\vee\to \C$.

\begin{proposition}
\label{prop:modcrit}
A formal Laurent series
\[
g(\tau)=\sum_{n\in \Q} a(n)\cdot q^n \in V((q))
\]
with coefficients $a(n)$ satisfying the conditions \eqref{eq:condt} and \eqref{eq:condz} is the $q$-expansion of a weakly holomorphic modular form in $M_k^!(\rho)$ if and only if
\[
\sum_{n\in \Q} (a(n),c(-n))=0
\]
for all
\[
f(\tau)=\sum_{n\in \Q} c(n)\cdot q^n \in M_{2-k}^!(\rho^\vee).
\]
\end{proposition}

\begin{proof}
This result is proved in Section 3 of \cite{Bo2} in the special case when $g$ is actually a formal power series. The same proof applies to our slightly more general case, if we replace the vector bundle of modular forms of type $\rho$ by a twist with a power of the line bundle $\calL(\infty)$ corresponding to the cusp at $\infty$.

Alternatively, we may replace the $q$-series $g$ by the $q$-series $g'=\Delta^j g$ for a positive integer $j$ such that $\Delta^j g$ is a power series. Here $\Delta$ is the normalized cusp form of weight $12$. Then one can literally apply \cite[Theorem 3.1]{Bo2} to $g'$ to deduce modularity in $M_{k+12j}(\rho)$ of this power series. Dividing out the power of $\Delta$ again, we obtain the result.
\end{proof}

\begin{proof}[Proof of Theorem \ref{thm:mod}]
According to Proposition \ref{prop:modcrit} with $\rho = \omega_L^\vee$, it suffices to show that
\begin{align}
\label{eq:keyrel}
\sum_{d\in \Q} \left( c(-d), [Z(d)]_\frakm\right)=0\in J^\ii_\frakm(X)_\C
\end{align}
for every
\begin{align}
f(\tau)=\sum_{d\in \Q} c(d)\cdot q^d \in M_{1/2}^!(\omega_L).
\end{align}
Since the space $M_{1/2}^!(\omega_L)$ has a basis of modular forms with integral coefficients \cite{McG}, it suffices to check that for every $f$ with integral coefficients the relation \eqref{eq:keyrel} holds. For $\mu\in L'$ we put
$c(d,\mu)=c(d)(\mu)$.

Let $\Psi(z,f)$ be the Borcherds lift of $f$ as in \cite[Theorem 13.3]{Bo1}. This is a meromorphic modular form on $\calD^+$ for the group $\Gamma$ of weight $c(0,0)$ with some multiplier system of finite order. Its divisor on $X$ is given by
\[
\dv(\Psi(z,f))=\sum_{d>0} \left(c(-d), Z(d)\right)+B(f),
\]
where $B(f)$ is a divisor of degree $c(0,0)/12$ supported at the cusps
of $X$.  Let $I\in \Iso(L)$. To determine the behavior of $\Psi(z,f)$ near the cusp $c_I$, we
identify $\calD^+$ with the upper complex half plane
$\H$ using \eqref{eq:dcoord}. Then $\Psi(w,f)$ has the infinite
product expansion
\begin{align}
\label{eq:BP}
\Psi(w,f)= R_I\cdot q_I^{\rho_I} \prod_{\substack{\lambda\in (L'\cap I^\perp)/I \\ n_I(\lambda) >0}}
\left(1-e^{2\pi i (\lambda,\ell')} q_I^{n_I(\lambda)}\right)^{c(-Q(\lambda),\lambda)},
\end{align}
which converges near the cusp $c_I$, that is, for $w\in \H$ with sufficiently large imaginary part.
Here the product runs over vectors $\lambda$ of negative norm which are associated to $I$, and
$\rho_I\in \Q$ is the Weyl vector at the cusp $c_I$ corresponding to $f$. Moreover, the quantity
$R_I$ is some constant in $k^\times$ of modulus $1$
times
\[
\prod_{\substack{a\in \Z/N_I\Z\\ a\neq 0}}(1-e^{2\pi i a/N_I})^{c(0,a\ell/N_I)/2}.
\]
Hence, the (finite) product
\[
\Psi(w,f)\times R_I^{-1}\prod_{\substack{\lambda\in (L'\cap I^\perp)/I \\ m_I>n_I(\lambda)>0}}
h_\lambda^{-c(-Q(\lambda),\lambda)}
\]
is a meromorphic modular form of weight $c(0,0)$ satisfying the condition \eqref{eq:scond} at $c_I$. There exists a degree zero divisor $D$ supported
on $S$
such that
the finite product
\[
\Psi(w,f)\times u_{D,(R_I)}^{-1}\times\prod_{I\in \Iso(L)/\Gamma}
\prod_{\substack{\lambda\in (L'\cap I^\perp)/I \\ m_I>n_I(\lambda)>0}}
h_\lambda^{-c(-Q(\lambda),\lambda)}
\]
is a meromorphic modular form of weight $c(0,0)$ satisfying
the condition \eqref{eq:scond} at all cusps
and having order $0$ at all cusps different from $c_{I_0}$.
Here
$u_{D,r}\in H_{\G_m,\frakm}$ denotes the
function defined in \eqref{eq:rdu}.

By the choice of the base point $s_0=c_{I_0}$ in \eqref{eq:cl},
the class of the line bundle $\calM_{c(0,0)}$ in $J_\frakm(X)$ is given by
\begin{align}
\label{eq:ls}
[\calM_{c(0,0)}]_\frakm &= [\dv(\Psi(f))-\deg(\calM_{c(0,0)})(c_{I_0})]_\frakm -[\dv(u_{D,(R_I)})]_\frakm \\
\nonumber
&\phantom{=}{}-\sum_{I\in \Iso(L)/\Gamma}\sum_{\substack{\lambda\in (L'\cap I^\perp)/I \\m_I>n_I(\lambda)>0}}
c(-Q(\lambda),\lambda) \cdot [\dv(h_{\lambda})]_\frakm.
\end{align}
Using Lemma \ref{lem:zd00}, we see that
\begin{align*}
&\sum_{I\in \Iso(L)/\Gamma}\sum_{\substack{\lambda\in (L'\cap I^\perp)/I \\m_I>n_I(\lambda)>0}}
c(-Q(\lambda),\lambda) \cdot [\dv(h_{\lambda})]_\frakm=-\sum_{d<0}(c(-d),[Z(d)]_\frakm).
\end{align*}
Inserting this into \eqref{eq:ls}, we obtain the relation
\begin{align*}
-c(0,0)[\calM_{-1}]_\frakm &= \sum_{d>0} \left(c(-d), [Z(d)]_\frakm\right)
+\sum_{d<0} \left(c(-d), [Z(d)]_\frakm\right)-[\dv(u_{D,(R_I)}) - D]_\frakm
\end{align*}
in $J_\frakm(X)_\C$.
This implies \eqref{eq:keyrel} in $J^\ii_\frakm(X)_\C$, concluding the proof of
the theorem.
\end{proof}

\begin{remark}
To be able to describe the generating series in $J_\frakm(X)_\C$ instead of in the quotient $J^\ii_\frakm(X)_\C$, we would have to know the
normalizing factors $R_I$ in \eqref{eq:BP} more precisely.
It would be very interesting to understand these better. Are they roots of unity?
\end{remark}

According to  the Manin-Drinfed theorem, the natural homomorphism
$J_\frakm(X)\to J(X)$ induces
a linear map
\[
J^\ii_\frakm(X)_\C\longrightarrow J(X)_\C.
\]
The classes $[Z(d)]_\frakm$ with $d\leq 0$ are in the kernel. Applying this map  coefficientwise to the generating series $A_\frakm$ in Theorem \ref{thm:mod}, we obtain the Gross-Kohnen-Zagier theorem.

\begin{corollary}[Gross-Kohnen-Zagier]
\label{cor:gkz}
The generating series
\begin{align*}
A_0(\tau)= \sum_{d>0}[Z(d)]_0\cdot q^{d}
\end{align*}
of the classes of the Heegner divisors in the Jacobian $J(X)_\C$ is the $q$-expansion of a cusp form in $S_{3/2}(\omega_L^\vee)\otimes J(X)_\C$.
\end{corollary}

\section{Traces of singular moduli}
\label{sect:5}

Here we show that every harmonic Maass form of weight zero
with vanishing constant terms
defines a linear functional of the generalized Jacobian
$J_\frakm^\ii(X)_\C$. Applying it to the generating series $A_\frakm$,
one obtains modularity results for traces of CM values of harmonic
Maass forms and weakly holomorphic modular forms as in \cite{Za2},
\cite{BF2}.

Let $H_k^+(\Gamma)$ be the space of harmonic Maass forms of weight $k$
for $\Gamma$ as in \cite[Section~3]{BF1}.
Recall that there is a surjective differential operator
$\xi_k:H_k^+(\Gamma)\to S_{2-k}(\Gamma)$ to cusp forms of `dual'
weight $2-k$.

For the rest of this section we fix a non-zero $F\in H_{0}^{+}(\Gamma)$. We denote the holomorphic part of the
Fourier expansion of $F$ at the cusp $c_I$ corresponding to $I\in \Iso(L)$ by
\begin{align}
\label{eq:Fexp}
F_I^+=\sum_{j\in \Z} c_{F,I}^+(j) \cdot q_I^j.
\end{align}
We define the order of $F$ at the cusp $c_I$ by
\[
\ord_{c_I}(F)= \min\{ j\in \Z: \; c_{F,I}^+(j)\neq 0\}.
\]

\begin{proposition}
\label{prop:psi}
\label{rem:psi}
Assume that for all $I\in \Iso(L)$ we have $\ord_{c_I}(F)>-m_I$ and $c_{F,I}^+(0)=0$.
\begin{enumerate}
\item[(i)] There is a linear map
\begin{align*}
\tr_F:J_\frakm(X)\longrightarrow\C
\end{align*}
defined by
\begin{align*}
[D]_\frakm\mapsto \tr_F(D):=\sum_{\substack{a\in \supp(D)\setminus S}} n_a \cdot F(a)
\end{align*}
for divisors $D=\sum_a n_a \cdot(a)$ in $\Div^0(X)$.
\item[(ii)]
The map $\tr_F$ vanishes on $H_{\G_m,\frakm}$ and factors through $J_\frakm^\ii(X)$.
\end{enumerate}
\end{proposition}

\begin{proof}
(i) We have to show that $\tr_F(D)=0$, for every divisor $D=\dv(g)\in P_\frakm(X)$ given by a rational function $g\in k(X)^\times$ satisfying
\[
q_I^{-\ord_{c_I}(g)}g=1+O(q_I^{m_I})
\]
at every cusp $c_I$. The expansion of the logarithmic derivative of $g$ with respect to the local parameter $q_I$ at $c_I$
is of the form
\[
\frac{d g}{g}
= \ord_{c_I}(g)q_I^{-1}  +O(q_I^{m_I-1}).
\]

If $F$ is weakly holomorphic, then $\eta:=F\frac{dg}{g}$ is a meromorphic $1$-form on $X$. Hence, by the residue theorem, the sum of the residues of $\eta$ vanishes, and we have
\[
\sum_{\substack{a\in X\setminus S}} \res_a(\eta) = -\sum_{a\in S}\res_{a}(\eta).
\]
The left hand side of this equality is given by $\tr_F(D)$, while the right hand side
satisfies
\begin{align*}
\sum_{a\in S}\res_{a}(\eta) &= \sum_{I\in \Iso(L)/\Gamma}\res_{c_I}(\eta)\\
&=
\sum_{I\in \Iso(L)/\Gamma}\res_{q_I=0}\left(\left(\ord_{c_I}(g)q_I^{-1}+O(q_I^{m_I-1})\right) \sum_{j>-m_I} c_{F,I}^+(j)\cdot q_I^j\right)=0.
\end{align*}
Here we have also used the fact that $c_{F,I}^+(0)=0$ for all $I$.

To prove the assertion for general $F\in H_{0}^{+}(\Gamma)$, we let $X_\eps$ be the manifold with boundary obtained from $X$ by cutting out small oriented discs of radius $\eps$ around the points in $\supp(\dv(g))\cup S$. Then for the $1$-form $\eta:=F\frac{dg}{g}$ it is still true that
\[
\lim_{\eps\to 0} \int_{\partial X_\eps} \eta= 0.
\]
Indeed, we have
\begin{align*}
\int_{\partial X_\eps} \eta &= \int_{\partial X_\eps} F\cdot \partial \log|g|^2\\
&= \int_{ X_\eps} d(F\cdot \partial \log|g|^2).
\end{align*}
Since $\log|g|^2$ and $F$ are harmonic functions on $X_\eps$, we find that
\begin{align*}
\int_{\partial X_\eps} \eta
&= \int_{ X_\eps} (\bar\partial F)\wedge(\partial \log|g|^2)\\
&= -\int_{ X_\eps} \partial\left((\bar\partial F) \log|g|^2\right)\\
&= -\int_{ \partial X_\eps} (\bar\partial F) \log|g|^2.
\end{align*}
In the latter integral, the differential $\bar\partial F=\overline{\xi_0(F)}d\bar z$ is antiholomorphic (hence smooth) on all of $X$. Since $\log|g|^2$ has only  logarithmic singularities, the integral vanishes in the limit $\eps\to 0$.

On the other hand, a local computation shows that
\begin{align}
\label{eq:h1}
\lim_{\eps\to 0} \int_{\partial X_\eps} \eta = \tr_F(D) + \sum_{I\in \Iso(L)/\Gamma}\res_{c_I} (F_I^+\cdot \frac{dg}{g}).
\end{align}
The vanishing of the second summand on the right hand side follows as before, proving that $\tr_F(D)=0$ again.

(ii) Let $u_{D,r}$ be as in \eqref{eq:rdu}. The same argument shows that $\tr_F(\dv(u_{D, r}))$ vanishes and $\tr_F$ factors through $J_\frakm^\ii(X)$.
%
\end{proof}

\begin{theorem}
\label{thm:tr}
Assume that for all $I\in \Iso(L)$ we have $\ord_{c_I}(F)>-m_I$ and $c_{F,I}^+(0)=0$.
Then $\tr_F(A_\frakm)$ is a weakly holomorphic modular form in $M_{3/2}^!(\omega_L^\vee)$, and
\begin{align*}
\tr_F(A_\frakm)(\varphi)&
=\sum_{d<0} \tr_F([Z(d,\varphi)]_\frakm)\cdot q^d
+\tr_F([\calM_{-1}]_\frakm)\varphi(0)
+\sum_{d>0} F\big(Y(d,\varphi)\big)\cdot q^d.
\end{align*}
Moreover, for $d<0$ the quantity $\tr_F([Z(d,\varphi)]_\frakm)$ is given by the finite sum
\[
\tr_F([Z(d,\varphi)]_\frakm)=
-
\frac{1}{2}\sum_{I\in \Iso(L)/\Gamma}\sum_{\substack{\lambda\in L_{d,I}'/\Gamma_{I}}}
\big(\varphi(\lambda)+\varphi(-\lambda)\big)
\cdot \sum_{j\geq 1}e^{2\pi i (\lambda,\ell_I')j}c_{F,I}^+\left(-n_I(\lambda) j\right).
\]
\end{theorem}

\begin{proof}
The modularity of $\tr_F(A_\frakm)$ is a direct consequence of Theorem \ref{thm:mod} and Proposition~\ref{prop:psi}.

We now compute the $q$-expansion.
For $d>0$ and $\varphi\in S_L$ we have by definition of the map $\tr_F$ that
\[
\tr_F([Z(d,\varphi)]_\frakm) = F(Y(d,\varphi)).
\]
If $d<0$ and $d\in-2\disc(L)(\Q^\times)^2$, we obtain
by the definition of the class $[Z(d)]_\frakm$ that
\begin{align}
\label{eq:fhn0}
\tr_F([Z(d,\varphi)]_\frakm)&=
\sum_{\substack{\lambda\in L_d'/\Gamma}}
\frac{1}{2n_I(\lambda)}(\varphi(\lambda)+\varphi(-\lambda)\big)
\cdot \tr_F\left(\dv \left(h^{-1}_{\lambda}\right) \right)
\\
\nonumber
&=-\sum_{I\in \Iso(L)/\Gamma}\sum_{\substack{\lambda\in L_{d,I}'/\Gamma_{I}}}
\frac{1}{2n_I(\lambda)}(\varphi(\lambda)+\varphi(-\lambda)\big)
\cdot F\left(\dv \left(h_{\lambda}\right) \right).
\end{align}
Arguing as in the proof of Proposition \ref{prop:psi}, in particular \eqref{eq:h1}, we
find  for $\lambda\in L'_{d,I}$ that
\begin{align}
\label{eq:fhn}
F(\dv(h_\lambda))
&=-\sum_{J\in \Iso(L)/\Gamma}\res_{c_J} (F_J^+\cdot \frac{dh_\lambda}{h_\lambda})\\
\nonumber
&= \res_{c_I}\left( F_I^+\cdot\frac{n_I(\lambda) \cdot e^{2\pi i (\lambda,\ell_I')} q_I^{n_I(\lambda)-1}+O(q_I^{m_I-1})}{1-e^{2\pi i (\lambda,\ell_I')}q_I^{n_I(\lambda)}}\right)\\
\nonumber
&=n_I(\lambda)\sum_{j\geq 1}e^{2\pi i (\lambda,\ell_I')j}c_{F,I}^+(-n_I(\lambda)j).
\end{align}
Inserting this into
\eqref{eq:fhn0}, we obtain the assertion.
\end{proof}

\begin{remark}
\label{rem:ct1}
The constant term $\tr_F([Z(0,\varphi)]_\frakm)$ can also be computed explicitly, see Propsition \ref{prop:ex} for an example.
\end{remark}

\subsection{An example}
\label{sect:5.1}

Here we consider as an example the modular curve $X_0(M)$
for a squarefree $M\in \Z_{>0}$.
Let $L$ be the lattice
\begin{align}
\label{eq:exl}
L=\left\{\zxz{b}{a/M}{c}{-b}:\; a,b,c\in \Z\right\}
\end{align}
with the quadratic form $Q(X)=M\det(X)$. Then $L'/L\cong \Z/2M\Z$ and $\SO^+(L)$ is isomorphic to the extension $\Gamma_0^*(M)$ of $\Gamma_0(M)$ by the Atkin-Lehner involutions. The discriminant kernel subgroup $\Gamma$ is isomorphic to $\Gamma_0(M)$, and
the modular curve $X_\Gamma$ is isomorphic to $X_0(M)$ with the cusp associated to the isotropic line  $I_0=\Z\kzxz{0}{0}{1}{0}$ corresponding to $\infty$, see e.g.~\cite[Section 2.4]{BO}.

The group $\Gamma_0^*(M)$ acts transitively on $\Iso(L)$, and the orbits are represented by the lines $I_D=W_D .I_0$ for the positive divisors $D\mid M$. Here $W_D\in \operatorname{PGL}^+_2(\Q)$ denotes the Atkin-Lehner involution with index $D$.
In particular, the set $S$ of cusps of $X_0(M)$ is in bijection with the set of positive divisors of $M$. If $I\in \Iso(L)$, we write $D_I$ for the unique positive divisor of $M$ such that $I$ is equivalent to $W_{D_I} .I_0$ under $\Gamma$.
Let $F\in H_0^+(\Gamma)$ be a harmonic Maass form.
The expansion of $F$ at the cusp $I_D$ as in \eqref{eq:Fexp} is given by the Fourier expansion  of $F\mid W_D$.

\begin{proposition}
\label{prop:ex}
Assume that for all $I\in \Iso(L)$ we have $\ord_{c_I}(F)>-m_I$ and $c_{F,I}^+(0)=0$.
The constant term of the generating series $\tr_F(A_\frakm)$
is given by
\[
\tr_F([\calM_{-1}]_\frakm)=2\sum_{D\mid M} \sum_{j\geq 1} c_{F,I_D}^+(-j)\cdot D\cdot \sigma_1(j/D).
\]
\end{proposition}

\begin{remark}
As shown in \cite[Remark 4.9]{BF2}, the right hand side above is also equal to $- \frac{1}{4\pi} \int_{\Gamma_0(M) \backslash \H }^{\mathrm{reg}} 
F \,d\mu$. 
The proposition gives a geometric interpretation of this regularized integral.
\end{remark}

\begin{proof}[Proof of Proposition \ref{prop:ex}]
We use the notation of the proof of Theorem \ref{thm:tr}.
By linearity it suffices to compute the class of the line bundle  $\calM_{12}$.
Since $X_{\Gamma}\cong X_0(M)$, a section of this line bundle is given by
the usual discriminant function $\Delta=q\prod_{n\geq 1}(1-q^n)^{24}$.
To compute the class of $\calM_{12}$ in the generalized Jacobian, we have to modify this section by multiplying with rational functions such that the local conditions \eqref{eq:scond} at the cusps are satisfied.
It is easily checked that
\[
\Delta\mid W_D= D^{-6}\Delta(D\tau) = D^{-6}q^{D}\prod_{n\geq 1}(1-q^{Dn})^{24}.
\]
This implies that the section
\[
s=\Delta\cdot
\prod_{I\in \Iso(L)/\Gamma}
\prod_{\substack{\lambda\in (L'\cap I^\perp)/I \\m_I>n_I(\lambda)>0}} h_{D_I\lambda}^{-24},
\]
has the expansion
\[
s=D^{-6} q_{I_D}^D \cdot \left(1+ O(q_{I_D}^{D m_I})\right)
\]
at the cusp $I_D$. For the $|S|$-tuple $r=(D^{6})_{D\mid M}$, and the function $u_{0,r}\in \Q(X_0(M))^\times$, the section $s\cdot u_{0,r}$ of $\calM_{12}$ satisfies the local conditions \eqref{eq:scond} at all cusps.
Therefore, in view of \eqref{eq:cl} and Proposition \ref{prop:psi} (ii), we have
\begin{align*}
\tr_F([\calM_{12}]_\frakm)&=\tr_F([\dv(s \cdot u_{0,r})-\deg(\calM_{12})\cdot(c_{I_0})]_\frakm)\\
&=-24\sum_{I\in \Iso(L)/\Gamma}\sum_{\substack{\lambda\in (L'\cap I^\perp)/I \\m_I>n_I(\lambda)>0}} F(\dv(h_{D_I\lambda})).
 \end{align*}
Using the formula \eqref{eq:fhn}, we get
\begin{align*}
\tr_F([\calM_{12}]_\frakm)
&=-24\sum_{I\in \Iso(L)/\Gamma}\sum_{n=1}^{m_I-1}D_I n\sum_{j\geq 1}c_{F,I}^+(-D_I nj)\\
&=-24\sum_{D\mid M}  \sum_{j\geq 1} c_{F,I_D}^+(-j)\cdot D\cdot \sigma_1(j/D).
 \end{align*}
This concludes the proof of the proposition.
\end{proof}

We now explain how to obtain a scalar valued generating series from $\tr_F(A_\frakm)$.
By means of the canonical pairing $(S_L,S_L^\vee)\to \C$, we define a map
\[
S_L^\vee\to \C,\quad u\mapsto (\chi_1,u)
\]
given by the pairing with the constant function $\chi_1$ with value $1$.
It induces a map from $S_L^\vee$-valued  to scalar valued modular forms,
\[
M^!_{3/2}(\omega_L^\vee)\longrightarrow M^!_{3/2}(\Gamma_0(4M)), \quad
f(\tau)\mapsto f^{\operatorname{scal}}(\tau) :=f(\chi_1)(4M\tau),
\]
see \cite[\S 5]{EZ}. The image lies in the Kohnen plus-space.
Applying this map to the generating series $A_\frakm$ of Theorem \ref{thm:mod}, we obtain a scalar valued generating series which has level $4M$. In particular, this implies
Theorem~\ref{thm:intromod} of the introduction.
If we apply this map to Theorem~\ref{thm:tr} and  use Proposition \ref{prop:ex}, we obtain:

\begin{theorem}
\label{thm:trex}
Let $L$ be as in \eqref{eq:exl}.
Assume that for all $I\in \Iso(L)$ we have $\ord_{c_I}(F)>-m_I$ and $c_{F,I}^+(0)=0$.
Then $\tr_F(A_\frakm^{\operatorname{scal}})\in M_{3/2}^!(\Gamma_0(4M))$, and
\begin{align*}
\tr_F(A_\frakm^{\operatorname{scal}})&=-
\sum_{D\mid M}\sum_{b\geq 1} \sum_{n\geq 1} c^+_{F,I_D}(-bn)\cdot  b \cdot q^{-b^2}
+2\sum_{D\mid M} \sum_{n\geq 1} c^+_{F,I_D}(-n)\cdot  D\cdot \sigma_1(n/D)\\
&\phantom{=}{}
+\sum_{\substack{d\in \Z_{>0}}} F\big(Y(d/4M,\chi_1)\big)\cdot q^d.
\end{align*}
\end{theorem}

When $M=p$ is a prime and $F$ is invariant under the Fricke involution, we obtain
Theorem \ref{thm:introtr} of the introduction. 

Now let $M=1$ and let $j=E_4^3/\Delta$ be the classical $j$-function. Write $J=j-744=q^{-1}+196884q+\dots$ for the normalized Hauptmodul for $\PSL_2(\Z)$  with vanishing constant term. Applying Theorem \ref{thm:trex} with $F=J$, we recover Zagier's original result \cite{Za2}:

\begin{corollary}
The generating series
\begin{align*}
-q^{-1} + 2+ \sum_{\substack{d\in \Z_{d>0}}} J(Y(d/4,\chi_1))\cdot q^{d}
\end{align*}
of the traces of singular moduli is a weakly holomorphic modular form for $\Gamma_0(4)$ of weight $3/2$ in the plus space.
\end{corollary}

\section{Generalizations}

In the section
we describe some variants of our main results and indicate possible generalizations.

\subsection{Modularity in the generalized class group}

In the definition of the Heegner divisors $Z(d)$ we have projected to degree $0$ divisors by subtracting a suitable multiple of $(c_{I_0})$. We now briefly describe what happens if we do not apply this projection and consider the divisors $Y(d,\varphi)$ defined in \eqref{eq:defheeg} for $d>0$. 
Then the corresponding generating series is a non-holomorphic modular form, where the non-holomorphic part is coming from a generalization of Zagier's weight $3/2$ Eisenstein series.

We let $\Cl_\frakm(X)$ be the generalized class group of $X$ with respect to the modulus $\frakm$, which we define as the quotient of the group of divisors on $X$ defined over $k$ modulo the subgroup $P_\frakm(X)$. Moreover, in analogy with \eqref{eq:defja} we put
\begin{align}
\Cl_\frakm^{\ii}(X)=\Cl_\frakm(X)/ H_{\G_m,\frakm}.
\end{align}
If $d>0$, we write $[Y(d,\varphi)]_\frakm$ for the class of the divisor $Y(d,\varphi)$ in $\Cl_\frakm(X)$. For $d=0$ we put $[Y(0,\varphi)]_\frakm=\varphi(0)[\calM_{-1}]_\frakm$ where the class in $\Cl_\frakm(X)$ of a line bundle $\calL$ is defined as
in \eqref{eq:cl} but without the summand $\deg(\calL)\cdot (s_0)$.
Finally, for $d<0$ we let $[Y(d,\varphi)]_\frakm = [Z(d,\varphi)]_\frakm$.

Recall from \cite[Theorem 3.5]{Fu} that there is a (non-holomorphic) weight $3/2$ Eisenstein series $E_{3/2,L}(\tau)$ whose coefficients with non-negative index are given by the degrees of the $Y(d,\varphi)$ (see also \cite{Ku:Integrals}). It is a harmonic Maass form of weight $3/2$ for the group $\Mp_2(\Z)$ with representation $\omega_L^\vee$ and generalizes Zagier's non-holomorphic Eisenstein series \cite{Za1}. Its Fourier expansion decomposes as
\[
E_{3/2,L}(\tau)=E_{3/2,L}^+(\tau)+E_{3/2,L}^-(\tau),
\]
where the holomorphic part is the generating series of the degrees of Heegner divisors,
\[
E_{3/2,L}^+(\tau)=\sum_{d\geq 0} \deg(Y(d))\cdot q^d,
\]
and the non-holomorphic part $E_{3/2,L}^-$ is a period integral of a linear combination of unary theta series. We obtain the following variant of Theorem~\ref{thm:mod}.

\begin{theorem}
The generating series
\begin{align*}
\label{eq:genser2}
\tilde A_\frakm(\tau)= \sum_{d\in \frac{1}{N}\Z}
[Y(d)]_\frakm\cdot q^{d} + E_{3/2,L}^-\cdot (c_{I_0})
\end{align*}
is a non-holomorphic modular form of weight $3/2$ for $\Mp_2(\Z)$ with representation
$\omega_L^\vee$ with values in $\Cl_\frakm^\ii(X)$.
Moreover, we have
\[
A_\frakm = \tilde A_\frakm - E_{3/2,L}\cdot (c_{I_0}).
\]
\end{theorem}

\subsection{Twists by genus characters}
Let $L$ be the lattice of Section~\ref{sect:5.1} for a squarefree $M\in \Z_{>0}$, and recall that $\Gamma\cong \Gamma_0(M)$.
For a discriminant $\Delta \neq 1$ and $r \in \Z$ such that $\Delta \equiv r^2 \bmod 4M$, we can define a generalized genus character $\chi_\Delta$ on $L'$ as in \cite[Section~I.2]{GKZ} and \cite[Section~4]{BO},
$$
\chi_\Delta(\lambda) =
\begin{cases}
   \leg{\Delta}{n},  & \text{if } \Delta \mid b^2 - 4Mac \text{ and } (b^2 - 4Mac)/\Delta \text{ is a square }\\
& \text{modulo } 4M \text{ and } \gcd(a, b,c , \Delta) = 1,\\
0, & \text{ otherwise,}
\end{cases}
$$
with $\lambda = \zxz{b/2M}{-a/M}{c}{-b/2M} \in L'$ and $n \in \Z$ any integer prime to $\Delta$ represented by one of the quadratic forms $[M_1 a, b, M_2 c]$ with $M_1 M_2 = M$.
Note that $\chi_\Delta$ is $\SO^+(L)$-invariant.

If $\lambda\in L'$ with $Q(\lambda)\in -4M(\Q^\times)^2$, let $I$ be the isotropic line associated with $\lambda$, and let $h_{\Delta, \lambda} \in \Q(\sqrt{\Delta})(X)^\times$ be a rational function with the following expansions
$$
h_{\Delta, \lambda} =
\begin{cases}
\displaystyle\prod_{b \in \Z/\Delta \Z}   \left(
1 - e^{2\pi i b/\Delta} q_{I}^{n_{I}(\lambda)}
\right)^{\leg{\Delta}{b}}
 + O(q_{I}^{m_{I}}),
& \text{ at the cusp } c_I,\\[3ex]
1 + O(q_J^{m_J}), & \text{ at all other cusps }c_J.
\end{cases}
$$
%
%
Suppose that $(\Delta, 2M) = 1$, or equivalently $(r, 2M) = 1$.
For each $d \in \frac{1}{4M} \Z$ and $\varphi \in S_L$, we can define the divisor $Z_{\Delta, r}(d, \varphi) \in \Div^0(X)_\C$ by
\begin{align*}
  Z_{\Delta, r}(d, \varphi) :=
  \begin{cases}
  \displaystyle
\sum_{\lambda \in L'_{d |\Delta|}/\Gamma }
\frac{1 }{2 |\Gamma_{\lambda}|}
\chi_\Delta(\lambda) \varphi(r^{-1}  \lambda)
\cdot
(z_\lambda) , & d > 0,\\[1ex]
\displaystyle \sum_{\lambda \in L'_{d/ |\Delta|} / \Gamma}
\frac{1}{2 n_I(\lambda)}\big( \varphi(r \lambda) + \mathrm{sgn}(\Delta) \varphi(-r \lambda) \big)
\cdot \mathrm{div}(h^{-1}_{\Delta, \lambda}) , & d \in -\frac{|\Delta|}{4M}(\Z_{>0})^2,\\[1ex]    \displaystyle 0,  & \text{otherwise.}
  \end{cases}
\end{align*}
All these divisors are defined over $\Q(\sqrt{\Delta})$ and have coefficients in the field of definition of $\varphi$.
We write $[Z_{\Delta, r}(d)]_\mf \in S^\vee_L \otimes J^\ii_\mf(X)$ for the element that sends $\varphi$ to $[Z_{\Delta, r}(d, \varphi)]_\mf \in J^\ii_\mf(X)$.
Define the representation $\tilde{\omega}_L$ to be $\omega_L$ if $\Delta > 0$ and $\overline{\omega}_L$ if $\Delta < 0$.
Then we have the following abstract modularity result.
\begin{theorem}
\label{thm:twist}
The generating series
$$
A_{\Delta, r, \mf}(\tau) := \sum_{d \in \frac{1}{4M}\Z} [Z_{\Delta, r}(d)]_\mf \cdot  q^d \in S_L^\vee ((q)) \otimes J^\ii_\mf(X)_\C
$$
is the $q$-expansion of a weakly holomorphic modular form in $M^!_{3/2}(\tilde{\omega}_L) \otimes J^\ii_\mf(X)_\C$.
\end{theorem}

This result comes out of calculating the effect of the intertwining operator in \cite[Section~3]{AE} applied to the generating series $A_\mf(\tau)$ associated to the scaled lattice $(\Delta L, \tfrac{Q(\cdot)}{|\Delta|})$.
The conditions that $M$ is square-free and $(\Delta, 2M) = 1$ are imposed to simplify the definition of $Z_{\Delta, r}(d, \varphi)$ and can be removed with a more complicated definition of the classes.
Note that it is necessary for $\mathrm{sgn}(\Delta)$, which determines the parity of $\tilde{\omega}_L$, to appear in the definition of $Z_{\Delta,r}(d,\varphi)$.
Alternatively, one could use the twisted Borcherds products in \cite[Theorem 6.1]{BO} to give a proof of Theorem \ref{thm:twist} along the same line as that of Theorem \ref{thm:mod} above.
By applying the functionals of Proposition \ref{prop:psi} to the twisted generating series of Theorem~\ref{thm:twist}, the main result
of \cite{AE} on twisted traces of harmonic Maass forms can be recovered.

\subsection{Other orthogonal Shimura varieties}
The Gross-Kohnen-Zagier theorem has been generalized 
to higher dimensional orthogonal Shimura varieties in \cite{Bo2}. 
Hence it is natural to ask whether our main results can also be generalized in the same direction. 
Let $L$ be an even lattice of signature $(n,2)$, and let $\Gamma$ be the discriminant kernel subgroup of $\SO^+(L)$.
Denote by $X_\Gamma$ a (suitable) toroidal compactification of the connected Shimura variety $Y_\Gamma$ associated to $\Gamma$. It would be interesting to define a generalized divisor class group as the group of divisors on $X_\Gamma$ modulo divisors of rational functions that satisfy certain growth conditions along the boundary of $X_\Gamma$.
 Is it possible to prove a modularity result analogous to Theorem \ref{thm:mod} for the classes of special divisors?
In this context, the product expansions obtained in \cite{Ku:BP}
with respect to one dimensional Baily-Borel boundary components may be helpful.

To illustrate this question, let us consider the easiest case for $n=2$ where the lattice $L$ is the even unimodular lattice of signature $(2,2)$.
Then the variety $X_\Gamma$ can be identified with  the product $X(1)\times X(1)$ of two copies of the compact modular curve of level $1$.
Special divisors on $X_\Gamma$ of positive index $d$ in the sense of \cite{Ku:Duke} are given by the Hecke correspondences $Z(d)$.
Let $q=(q_1,q_2)$ be the usual local coordinates near the boundary point $s=(\infty,\infty)\in X_\Gamma$. Let $m$ be a positive integer, and put
$\frakm=m\cdot (s)$.
If $k=(k_1,k_2)\in \Z^2$ we briefly write $q^k=q_1^{k_1}q_2^{k_2}$, and
for a meromorphic function $f$ in a neighborhood of $(\infty,\infty)$ we
write
$f=O(q^m)$ if in the Taylor expansion of $f$ at $(\infty,\infty)$  only terms of total degree at least $m$ occur.

Let $\Div_\frakm(X_\Gamma)$ be the free abelian group generated by pairs $(D,g_D)$, where $D$ is a prime Weil divisor on $X_\Gamma$ and $g_D$ is a local equation for $D$ in a small neighborhood of $s$. The local equations give rise to local equations $g_D$ near $s$ for arbitrary Weil divisors $D$.  Let $P_\frakm(X_\Gamma)$ be the subgroup of pairs $(D,g_D)\in \Div_\frakm(X_\Gamma)$ for which $D=\dv(f)$ is the divisor of a meromorphic function $f$ satisfying
\[
f\cdot g_D^{-1} = 1+O(q^m)
\]
near $s$. We define a generalized  class group as the quotient
$\Cl_\frakm(X_\Gamma)= \Div_\frakm(X_\Gamma)/P_\frakm(X_\Gamma)$.
It would be interesting to define suitable classes of special divisors in $\Cl_\frakm(X_\Gamma)$ of arbitrary integral index $d$ and to prove a modularity result for the generating series of these classes.

\end{document}